\documentclass[12pt]{amsart}  

\usepackage{amssymb} 
\usepackage{enumerate, amsfonts, latexsym}

\usepackage{amsthm}
\theoremstyle{plain}

\usepackage[all]{xy}
\usepackage{graphicx}
 
\newtheorem{thmA}{Theorem} 
 
\newtheorem{corA}[thmA]{Corollary}

\newtheorem{theorem}{Theorem}[section] 
\newtheorem{thm}[theorem]{Theorem}  
\newtheorem{cor}[theorem]{Corollary} 
 
\newtheorem{prop}[theorem]{Proposition}   
\newtheorem{lemma} [theorem]{Lemma}

\newtheorem{notation}[theorem]{Notation} 
 
\theoremstyle{definition} 
\newtheorem{definition}[theorem]{Definition} 
\newtheorem{example}[theorem]{Example} 
 
\theoremstyle{remark} 
\newtheorem{remark}[theorem]{Remark}

\numberwithin{equation}{section} 
 
\newfont{\msb}{msbm10 scaled 1200} 
\newfont{\euf}{eufm10 scaled 1200}    
\def\A{\mathcal A}
\def\B{\mathcal B}
  
\def\e{\varepsilon} 
\def\F{\mathcal F}
\def\G{\Gamma}     
\def\N{\mathbb N} 
\def\P{\mathcal P}  
\def\Q{\mathcal Q}
\def\R{\mathcal R}
\def\S{\Sigma} 
\def\T{\mathcal T} 
\def\Z{\mathbb Z}    
\def\U{\mathcal U}

\def\a{\alpha}
\def\r{\rho}

\def\t{\tau}

\def\ssm{\smallsetminus} 
\def\fib{\mathcal F}

\def\ar{\text{\rm{Area}}} 
\def\area{\text{\rm{Area}}}   
\def\arP{\area_\P}

\def\freely{\ {\underset{\hbox{free}}=}\ }

\def\ac{{\hbox{\rm{AC}}}} 
\def\acc{\hbox{\rm{AC}}_*} 

\def\ac{\hbox{\rm{AC}}}
\def\Fn{F(x_0,\dots,x_n)}

\def\<{\langle} 
\def\>{\rangle} 
\def\|{{\,|\! |\, }}

\def\ARnn{\<a_1,\dots,a_n\mid r_1,\dots,r_n\>} 
\def\AR{\<\mathcal A\mid\mathcal R\>}

\def\ker{\text{\rm{ker }}}    
\def\inv{^{-1}} 
\def\-{\underline}

\catcode`\@=11 
\def\serieslogo@{\relax} 
\def\@setcopyright{\relax} 
\catcode`\@=12

\begin{document} 
   
\title[Complexity of Balanced Presentations]
{The complexity of balanced presentations and the Andrews--Curtis conjecture} 
 
\author[Martin R. Bridson]{Martin R.~Bridson} 
\address{Martin R.~Bridson\\
Mathematical Institute\\
Andrew Wiles Building, ROQ\\ 
Woodstock Road\\
Oxford OX2 6GG} 
\email{bridson@maths.ox.ac.uk}   
 

\thanks{This work was supported by a Senior Fellowship from the EPSRC and a Wolfson Research Merit Award from the
Royal Society.}
 
   
   
\begin{abstract} Motivated by problems in topology,
we explore the complexity of balanced group presentations. We obtain large lower
bounds on the complexity of Andrews-Curtis trivialisations, beginning in rank 4. 
Our results are
based on a new understanding of how Dehn functions of groups behave
under  certain kinds of push-outs. 

We consider groups $S$ with  presentations
of deficiency 1 satisfying certain technical conditions and
construct balanced group presentations $\P_w$
indexed by words $w$ in the generators of  $S$.  
If $w=1$ in $S$ then $\P_w$  is Andrews-Curtis trivialisable and the number of
Andrews-Curtis moves required to trivialise it can be bounded above and below in terms of 
how hard it is to prove that $w=1$ in $S$.
\end{abstract}

\maketitle 

\centerline{\em For Andrew Casson, with admiration and respect}
   

\section{Introduction} A presentation of a group is {\em{balanced}} if it has the same number of generators and relators.
The study of balanced group  presentations 
bristles with famous open problems.
Balanced presentations of the trivial group hold a particular fascination because of their intimate connection to 
famous open problems in low-dimensional topology: most directly,  the
{\em{Andrews-Curtis Conjecture}} \cite{AC}, with its relation to the smooth
4-dimensional Poincar\'e conjecture \cite{kirby}; also
 the Zeeman conjecture, which asserts that if $K$ is a
 finite contractible 2-complex, then $K\times [0,1]$ is collapsible after subdivision \cite{Z}, \cite{metz1}; 
 the dimension of expansion required in Whitehead's simple-homotopy theorem \cite{jhc}; 
and the Contractibilty Problem for finite 2-complexes.

This last problem is equivalent to 
a problem highlighted by Magnus:
is there  an algorithm that can recognise whether a
balanced presentation describes the trivial group or not? If there is no such algorithm then,
in the light of \cite{freed}, there would be no algorithm that could
recognise the 4-sphere.

In this article we shall focus almost entirely on the algebraic formulation of the {\em{Andrews-Curtis conjecture}},
establishing by means of explicit constructions
large (but computable) lower bounds on the number of Andrews-Curtis moves that are required in order to
trivialise a balanced presentation of the trivial group.  In a 
sequel to this paper I shall discuss in details the topological and geometric consequences of these results.

The Andrews-Curtis conjecture asserts that any balanced presentation of the trivial group
$\P\equiv\<a_1,\dots,a_k\mid r_1,\dots,r_k\>$ can be reduced to the trivial presentation
$\mathbb{I}_k \equiv \<a_1,\dots,a_k\mid a_1,\dots,a_k\>$ by repeatedly applying to the list $(r_1,\dots,r_k)$
the following three AC-moves:
replace some $r_i$ by its inverse $r_i^{-1}$; replace $r_i$ by $r_ir_j$; or replace
$r_i$ by $ur_iu^{-1}$, where $u$ is any word in the free group $F$ on $\{a_1,\dots,a_k\}$. (Throughout,  
words  are regarded as elements of $F$, so free reduction and expansion are permitted.)
One says that $\P$ is {\em AC-trivialisable} if it satisfies this conjecture. The {\em stable
Andrews-Curtis conjecture} is weaker: one is allowed to replace $\<a_1,\dots,a_k\mid r_1,\dots,r_k\>$ 
by $\<a_1,\dots,a_k,x_1,\dots,x_\ell\mid r_1,\dots,r_k,x_1,\dots,x_\ell\>$ at the beginning of the reduction
process (any $\ell\in\N$) and the assertion of the conjecture is that for some $\ell$
one can perform AC-moves to reduce to $\mathbb{I}_{k+\ell}$. The results of this paper are valid regardless of
whether one allows stabilisation or not.
It will be convenient to count 
$r_j\mapsto r_j(ur_i^{\pm 1}u^{-1})$ as a single move, and we use the
term {\em dihedral AC-move} to mean this or one of the basic moves.
The minimum number
of dihedral AC-moves required to trivialise $\P$ will be denoted $\acc(\P)$. 

To give a clear sense of what is achieved in this article, let me begin with a simply stated special case of 
the Main Theorem.

\begin{thmA}\label{thmA} For $k\ge 4$ one can construct explicit
sequences of $k$-generator balanced presentations $\P_n$ of the 
trivial group so that 
\begin{enumerate}
\item the presentations $\P_n$ are AC-trivialisable;
\item the sum of the lengths of the relators in $\P_n$
is at most $24(n+1)$;
\item the number of (dihedral) AC moves required to trivialise $\P_n$ is bounded below by the function $\Delta(\lfloor \log_2 n\rfloor)$ where
$\Delta:\N\to\N$ is defined recursively by $\Delta(0)=2$ and $\Delta(m+1)=2^{\Delta(m)}$. 
\end{enumerate}
\end{thmA}

This special case of our Main Theorem is sufficient to illustrate an important
point: as a function of the sum of the lengths of the relators, 
the number of AC-moves required to trivialise $\P_n$ grows more quickly than any tower of exponentials; in particular,
it quickly exceeds the number of electrons in the universe. Thus it is physically impossible to exhibit 
an explicit sequence of AC-moves trivialising rather small balanced presentations of the trivial group. An
explicit example is given in Section \ref{example}.

My initial
work on this problem was inspired by conversations with Andrew Casson in which he explained his intriguing work
on the Andrews-Curtis conjecture and developments of Stallings' approach to the Poincar\'e conjecture \cite{casson}, \cite{stall}.
Among many other things, Casson (also \cite{mm}) found that in rank 2 one needs surprisingly 
few AC-moves to trivialise small balanced presentations of the trivial group. At the same time, various researchers were using computer experiments to probe
potential counterexamples to the Andrews-Curtis conjecture, e.g.  \cite{BKM}, \cite{havas}, \cite{mias}.
Theorem \ref{thmA} shows that one has
to exercise extreme caution when interpreting an experiment that fails to find a trivialising sequence of AC-moves.

Although we have postponed a full explanation of the topological consequences of this work to another paper, it is
worth noting here that lower bounds of the type established in Theorem \ref{thmA} immediately translate into lower
bounds on the complexity of various topological problems. For example, if $K_n$ is the standard\footnote{the 
2-complex that has one 0-cell, has 1-cells in bijection with the generators  of $\P_n$,
oriented and labelled,
and 2-cells in bijection with the relators $r$ of $\P_n$, with the 2-cell
corresponding to $r$ attached along the oriented loop in the 1-skeleton labelled $r$} 2-complex
for $\P_n$, then the 3-complex $L_n=K_n\times[0,1]$ will have a collapsible subdivision (see \cite{metz1}), but
while the number of cells in $L_n$ is bounded by a constant times $n$, the number of cells in any collapsible
sub-division is bounded below by $\Delta(\lfloor \log_2 n\rfloor)$. This is 
in sharp contrast to what happens when $K$ is the spine of a 3-manifold: in that case
one can compute upper bounds on the number of cells in a collapsible
subdivision from Perelman's solution to the 3-dimensional Poincar\'e  conjecture. Similarly, from Theorem \ref{thmA} one
obtains lower bounds on the complexity of sequences of handle slides 
bringing the 4-sphere obtained as the boundary of a regular neighbourhood of $K_n\hookrightarrow\mathbb{R}^5$ into 
standard form.

A technical innovation behind the lower bounds established in this paper is a device for encoding into balanced
presentations the complexity of the word problem in groups of a certain type.  This idea has many roots in the
field of decision problems, where the complexity of one problem has often been translated into another setting by
encodings that rely on well-controlled colimits in the category of 
groups, such as HNN extensions and amalgamated free products. (See \cite{bfs}, \cite{cfm1}, \cite{cfm2},
\cite{sapir}  and references therein.)
 Thus, for example, to parlay the existence of a finitely presented group $G$ with unsolvable word problem into a proof
that the triviality problem for finitely presented groups is unsolvable, one builds
a sequence of finite presentations $\P_w$, indexed by words $w$ in the
generators of $G$, so that $|\P_w|$ is the trivial group if and only if $w=1$ in $G$.  
Crudely speaking, this is the template that we want to follow here. But our situation is 
more subtle: we have to
arrange for the presentations $\P_w$ to be balanced, and this constrains us greatly. Moreover,
rather than dealing simply with (un)solvability, we have
to quantify and trace the complexity of the problems at hand.

The most natural measure of complexity for the word problem
of a finitely presented group, when the problem is tackled without extrinsic information, is the {\em{Dehn function}} of the group. Roughly speaking, if a word $w$ in the generators of $G$ equals $1\in G$, then
 $\area(w)$ is the number of relators that one has to apply to prove that $w=1$, and the {\em{Dehn function}} of
 $G$ is defined to be
 $\delta(n) = \sup\{\area(w): |w|\le n, \, w=_G1\}$, where $|w|$ denotes word-length.
 In our setting, we are forced to consider
a modification of this function which measures the length of the shortest
proof that some power of $w$ equals the identity in $G$, that is $\area^*(w) = \inf_{n\neq 0}\area(w^n)$.
The rudiments of the theory of Dehn functions will be recalled in Section \ref{ss:dehn}. The initial
definitions belie the fact that this is largely a geometric subject, revolving around the
study of van Kampen diagrams.

In order to prove results such as Theorem \ref{thmA}, we shall construct AC-trivial presentations
$\langle \a_1,\dots, \a_k\mid \r_1,\dots,\r_k\rangle$ with the property that
the algebraic area $\ar (\a_i)$ of each generator, measured as a function of
$\Sigma |\r_i|$, is huge; this suffices because, by  Lemma \ref{l:fib},
   the 
number of Andrews-Curtis moves required to trivialise the presentation is an
upper bound on the logarithm of $\ar (\a_i)$. Note the subtlety of what we are trying to do here:
we seek lower bounds on (a precise measure of) the complexity
of the word problem in groups that we know to be {\em{trivial}}.
We cannot use standard results relating the Dehn functions of HNN extensions and amalgamated free
products to the Dehn functions of their vertex groups; rather, we have to control the way in which the word problem
degenerates when we form colimits of diagrams of groups where the morphisms in the diagram are not injective. The
novel techniques for doing this are presented in Section \ref{s:area}, where the main result is:

\begin{thmA}\label{t:Area} Consider a finite presentation $\P\equiv\AR$, fix  $a_1\in\A$ and $u_0,u_1\in F(\A)$,  suppose that
$\langle a_1\rangle\cap \langle u_i\rangle =\{1\}$ in $|\P|$ for $i=0,1$,
and   $a_1$ has infinite order.
Let
$$\T = \langle \A,\, \hat\A,\, t,\, \hat t \mid \R, \, \hat \R,\, 
t^{-1}u_0tu_1^{-1},\, 
\hat t^{-1}\hat u_0\hat t\hat u_1^{-1},\, a_1\hat t^{-1},\, \hat a_1t^{-1}\rangle.
$$
Then, for all $v\in F(\A)$, 
$$
\ar_\T(v) \ge \min\{\arP(v),\, \arP^*(u_0),\, \arP^*(u_1)\}.
$$
\end{thmA}

Turning to the Main Construction, we fix a finite alphabet $\A=\{a_0,\dots,a_n\}$ and interpret words in the
free group $F(\A)$ as elements in a  {\em{seed group}} (with large Dehn function)
 that admits a  presentation   $\P\equiv \<a_0,a_1,\dots,a_n\mid r_1,\dots, r_n\>$ satisfying
a condition for which we need the following notation:
given a word $w$ in the letters $\A^{\pm 1}$,
we write $\-w$ for the word obtained from $w$ by deleting all occurrences of $a_0^{\pm 1}$.
Define
$$
\P_w\equiv
\< \A, \hat\A \mid \R,\, \hat\R,\,\hat a_1a_0\hat a_1^{-1}w\inv ,\, a_1\hat a_0 a_1^{-1}\hat w\inv \>
$$
and let $\G_w=|\P_w|$ be the group presented by $\P_w$.

\begin{thmA}\label{t:If} 
If
$\-\P\equiv \<a_1,\dots,a_n\mid \-r_1,\dots, \-r_n\>$ is a presentation of the trivial group
and $a_1$ has infinite order in $|\P|$, then,  for all words $w$ in the letters $a_i^{\pm 1}$, 
\begin{enumerate}
\item $\G_w$ is trivial if  $w=1$ in $|\P|$.
\item If $w=1$ and $\-\P$ is AC-trivialisable, then
$\P_w$ is AC-trivialisable  and 
$$\log\area^*_\P(w)\le
 \acc(\P_w) + 1.$$
\end{enumerate}
\end{thmA}

Theorem \ref{thmA} is obtained from Theorem \ref{t:if}
by taking $\P$ to be the natural presentation of $S_2$,
where $S_k$ is the much-studied 1-relator group 
$S_k= \langle x,\t \mid (\t x \t\inv)x (\t x \t\inv)\inv = x^k\rangle$. Here, $\t$ plays the r\^ole of $a_0$.  
The following theorem is proved in Section \ref{s:seed}.

\begin{thmA} \label{ThmS} There exists a sequence of words $w_n\in F(x,\tau)$
with lengths $|w_n|\le 12n$ such that $w_n=1$ in $S_k$ and
$\ar_{\S_k}^*(w_n) \ge \Delta_k(\lfloor\log_2 n\rfloor)$.
\end{thmA} 

Finally, we complement Theorem \ref{t:If} with an upper bound, thus completing the proof of our Main Theorem.

\begin{thmA} \label{t:Main} Let $\P\equiv\<a_0,a_1,\dots,a_n\mid r_1,\dots, r_n\>$. Assume that
$\<a_0\>$ and $\<a_1\>$ are infinite  and intersect trivially in $|\P|$. Let $\-\P$ and $\G_w$
be as defined previously. We consider
words $w$ such that $\<w\>\cap\<a_1\> = \{1\}$ in $|\P|$ and $\<w\>$ is trivial or infinite.
\begin{enumerate}
\item[{\bf (1)}] $\G_w=|\P_w|$ is trivial if and only if $w=1$ in $|\P|$.\smallskip

\item [{\bf (2)}] Suppose $w=1$ in $|\P|$. If $\-\P$ is AC-trivialisable then so is
$\P_w$ and, writing $|r|_0$ for the number of occurrences of $a_0$ in $r$,
$$\log\area^*_\P(w)\, -\, 1\ \le\ 
\acc(\P_w)\ \le\ 2\acc(\-\P) + 
2\area^*_\P(w)+2\sum_{r\in\R}|r|_0$$
\end{enumerate}
\end{thmA}

When the presentations $\P_{w_n}$ in Theorem \ref{thmA} are built from the 
words $w_n$ in Theorem \ref{ThmS},  the upper bound in Theorem \ref{t:Main}(2) complements the
lower bound in Theorem \ref{thmA}(2).

\begin{corA} One can construct the presentations $\P_n$ in Theorem \ref{thmA} so that
$n\mapsto\acc(\P_n)$ is $\simeq$ equivalent to $\Delta(\lfloor \log_2 n\rfloor)$.
\end{corA}

We have focussed on Andrews-Curtis complexity in this introduction, but from a technical point of view the
key measure of how hard it is to prove that $|\P_w|\cong\{1\}$ is the sum of $\ar(a)$, as $a$ runs over the
generators of $\P_w$. We used Lemma \ref{l:fib} to parlay this into lower bounds on $\acc(\P_w)$, but we could
equally have translated our estimates into lower bounds on the number of {\em Tietze moves} required to trivialise $\P_w$
(see Proposition \ref{p:tietze}). 
Finally, I should point out that although our results show that one needs huge numbers of elementary moves to
trivialise balanced presentations, this does not preclude the possibility that one may be able
to decide the {\em{existence}} of AC-trivialisations quickly. Indeed, in Section \ref{ss:polynomial} we shall see that
for presentations $\P_w$ parametrised by words in the generators of $S_2$, one can determine AC-triviality in polynomial time.

\medskip

\noindent{\em The lives of this article:}
I proved the results described in this paper in 2003 and presented them at the Arkansas Spring Lecture Series meeting ``The Andrews-Curtis and the Poincar\'e Conjectures", at which Andrew Casson was the Principal Speaker. I
wrote this paper on my return to London but delayed publishing it because I wanted to  include a full
account of the topological and geometric 
consequences of these results, and I was never satisfied with my attempts to do this.
Although I have given many detailed lectures on this material around the world since 2003, I know that my
failure to publish a definitive version of the paper has frustrated many colleagues and students: I apologise sincerely for this and thank them for their enduring interest and patient correspondence.

\section{Elementary Moves and Complexity}\label{s:moves}

We want to quantify the difficulty of decision problems, relating the complexity of word problems to the complexity
of trivialisation problems. In each context, we  count  how many elementary moves are required. In
the case of the word problem, this leads to a discussion of area and Dehn functions. For the trivialisation
problem, we work with AC-moves and Tietze moves. Proposition \ref{p:tietze} relates area to trivialisation moves.
We shall return to a 
discussion of how to interpret complexity in Section \ref{s:last}. 

\subsection{Moves on Balanced Presentations}\label{ss:defns}

The {\em  Andrews-Curtis conjecture} asserts that any balanced presentation of the trivial group
$\P\equiv\<a_1,\dots,a_k\mid r_1,\dots,r_k\>$ can be transformed to the trivial presentation
$\mathbb{I}_k \equiv \<a_1,\dots,a_k\mid a_1,\dots,a_k\>$ by repeatedly applying to the list $(r_1,\dots,r_k)$
the following three AC-moves:
\begin{enumerate}
\item[{\bf{(AC1).}}] replace some $r_i$ by its inverse $r_i^{-1}$
\item[{\bf{(AC2).}}] replace $r_i$ by $r_ir_j$, some $i\neq j$
\item[{\bf{(AC3).}}] replace $r_i$ by $ur_iu^{-1}$, with $u$ any  word  in the free group $F$ on $\{a_1,\dots,a_k\}$.
\end{enumerate} 
Throughout,  
words  are regarded as elements of the free group $F$, so free reduction and expansion are permitted.

The {\em stable
Andrews-Curtis conjecture} allows the additional stabilisation move
\begin{enumerate}
\item[{\bf{(AC4).}}] For some $\ell\in\N$, replace $\<a_1,\dots,a_k\mid r_1,\dots,r_k\>$ 
by $$\<a_1,\dots,a_k,x_1,\dots,x_\ell\mid r_1,\dots,r_k,x_1,\dots,x_\ell\>.$$
\end{enumerate}
There is no loss of generality in assuming that stabilisation is performed before the other moves. 
The stable conjecture asserts that some augmented presentation can be transformed
to $\<a_1,\dots,a_k,x_1,\dots,x_\ell\mid a_1,\dots,a_k,x_1,\dots,x_\ell\>$ by a sequence of moves of type (AC1), (AC2) and (AC3). 

\smallskip

{\em{Tietze's Theorem}} tells us that {\em{every}} presentation of the trivial group can be 
transformed to a trivial presentation if the following  move is permitted in addition to (AC1) to (AC4):
\begin{enumerate}
\item[{\bf{(T).}}]  the list $(r_1,\dots,r_k)$ can be edited by the insertion or deletion of empty relations
$$ (r_1,\dots,r_k) \leftrightarrow (r_1,\dots,r_k,\emptyset,\dots,\emptyset).$$
\end{enumerate}
There is no loss of generality in assuming that the empty relations are all added at the beginning of the
process and deleted at the end.

\begin{remark} The topological effect of the stabilisation move (AC4)  is to augment the standard 2-complex $K(\P)$
by attaching to it 2-discs at the vertex. If we have fixed an embedding $K(\P)\hookrightarrow\mathbb{R}^5$, this
can be done without altering the boundary of a regular neighbourhood of $K(\P)$. The Tietze move (T) is more
dramatic: for each empty relation added, a 2-sphere is attached to $K(\P)$ at the vertex, and each
sphere alters the boundary of a regular neighbourhood by a connected sum with $S^2\times S^2$ (cf.~\cite{wall2}).
\end{remark}

It will be convenient to count 
$r_j\mapsto r_j(ur_i^{\pm 1}u^{-1})$ as a single AC-move, and we use the
terminology {\em dihedral AC-move} to indicate that
we are allowing this.
The minimum number
of dihedral AC-moves required to trivialise   $\P$ is denoted $\acc(\P)$.

\subsection{Dehn functions and van Kampen diagrams}\label{ss:dehn}
 The reader unfamiliar with this material should consult the survey
\cite{bfs} for a thorough introduction (also \cite{gromov}, \cite{BRS}, \cite{olsh}).

The Dehn function of a finitely presented
group $G=\AR$
measures the complexity of the word problem by counting the number of 
relators that must be applied in order to reduce each word $w$ with $w=_G1$
to the empty word. To {\em apply a relator $r$} to a word $w$ means that $r$ and $w$ can be broken into
(perhaps empty) subwords $r\equiv  u_1u_2u_3$ and $w\equiv  \alpha u_2^{\pm 1} \beta$
and we replace $w$ by $\alpha (u_3u_1)^{\mp 1}  \beta$. One defines $\area(w)$ to be
the least number of relators that must be applied in order to reduce
$w$ to the empty word; one allows\footnote{For purposes of comparison with the
counting of Tietze and Andrews-Curtis moves,  one should note that
if one were to count each such  insertion or deletion 
as a move, then the resulting notion of Dehn function would be $\simeq$ to the
standard one --- see \cite{bfs}.}  free reduction by 
removal or insertion of   inverse pairs $aa\inv \leftrightarrow\emptyset$
between applications of relators. The {\em Dehn function} $\delta:\N\to\N$
of $\AR$ is then defined by
$$
\delta(n) =\max\{\area(w) \mid |w|\le n,\, w=_G1\}.
$$
In this article we shall also need to consider the quantity
$$
\area^*(w) := \min\{\area(w^n) \mid n\text{ a non-zero integer}\}.
$$

A basic lemma in the subject (see \cite{bfs} p.35) shows that when finite presentations define isomorphic groups, their
Dehn functions are $\simeq$ equivalent: by definition, $f\simeq g$ if  
$f\preceq g$ and $g\preceq f$, where $f\preceq g$
means that there exists a constant $C>0$ such that $f(n) \le C\, g(Cn+C)
+ Cn+ C$ for all $n\in\N$.  

When one comes to calculate with Dehn functions, it is useful to note that $\area(w)$
is the least integer $N$ such that there is an equality in the free group $F(\A)$
$$
w= \prod_{i=1}^N x_i^{-1}r_ix_i
$$
with $x_i\in F(\A)$ and $r_i\in\R^{\pm 1}$.

According to van Kampen's Lemma (see, e.g.~\cite{bfs} p.48, \cite{LS} or \cite{olsh}) the above factorisation of $w$ can be portrayed by a {\em van Kampen diagram}, which is a finite, contractible, planar, combinatorial 2-complex with a basepoint
at a vertex on the boundary.
The oriented 1-cells of this diagram are labelled by elements of $\A$, the boundary label on each face of the diagram is an element of $\R$
(or its inverse), and the boundary cycle of the complex (read with positive orientation from the basepoint) is the word $w$. The number of 2-cells in the diagram is equal to $N$, the number of factors in the given equality for $w$. Conversely, any van Kampen diagram gives rise to an equality in $F(\A)$ showing that the boundary label of the diagram represents the identity in $G$. A {\em minimal-area van Kampen diagram} is one that has the least number of 2-cells among all van Kampen diagrams which share its boundary label. 

Most techniques for obtaining lower bounds
on Dehn functions use van Kampen diagrams for $w$, and the bounds we need in
Section \ref{s:seed} follow this pattern.

\begin{example}\label{ex1} Consider the standard presentation $\Q_1\equiv\<x,y \mid xyx^{-1}y^{-1}\>$ 
for $\Z^2$, and let $c_{n,m}:= x^ny^mx^{-n}y^{-m}$. The universal cover of the standard 2-complex $K(\Q_1)$
is the square tiling of the Euclidean plane with horizontal edges labelled $x$ and vertical edges labelled $y$.
 The word $c_{n,m}$ labels a rectangular loop based at the origin and the rectangular
disc that it encloses is a minimal-area van Kampen diagram. (Reasoning in an elementary manner to prove  that
it is indeed least-area leads one in the direction of Lemma \ref{aspher}.) Also, $\ar^*(c_{n,m}) = \ar(c_{n,m})$.
\end{example}

\begin{example}\label{ex2} Consider the standard presentation $\Q_2\equiv\<a,s \mid s\inv as a^{-2}\>$ 
for the metabelian Baumslag-Solitar group ${\rm{BS}}(1,2)$, and let $u_n:= a^{2^n}s^{-n}a^{-1}s^n$.
A van Kampen diagram $D_n$ for $u_n$ is shown in figure \ref{f:bs}.
 The boundary label on each 2-cell, read
anticlockwise from the bottom right corner, is $a^2s\inv a^{-1}s$. 
Note that a more metrically appropriate picture of this diagram would render it as a dramatically flaring wedge rather
than a rectangle.

We now consider a second van Kampen diagram, with twice the area, obtained from $D_n$ as follows: take a second
copy of $D_n$,
obtained by reflection in the line labelled $a^{2^n}$, then shift this reflected copy down one notch so as to align its second-lowest vertex with the bottom of $D_n$. The two copies of $D_n$, arranged thus, form a van Kampen diagram with boundary
label $U_n:=as^{-n}a^{-1}s^na^{-1}s^{-n}as^n$.
 
As in Example \ref{ex1}, it is easy to believe that this diagram has least area among all van Kampen diagrams
for $U_n$.  Lemma \ref{aspher} provides a useful tool for proving this:
arguing topologically or with HNN normal form, one shows that the diagram embeds in the universal cover of
$K(\Q_2)$. 

 Counting only those 2-cells along the line labelled $a^{2^n}$ gives a lower bound of $2^n$ on the
area of $U_n$. This rather crude estimate actually gives a sharp lower bound (up to $\simeq$ equivalence) on the
Dehn function of $\Q_2$. An entirely similar analysis applies to the natural presentation of ${\rm{BS}}(1,k)$.
\end{example}

\begin{figure}
\begin{center}
\includegraphics[width=3.5in]{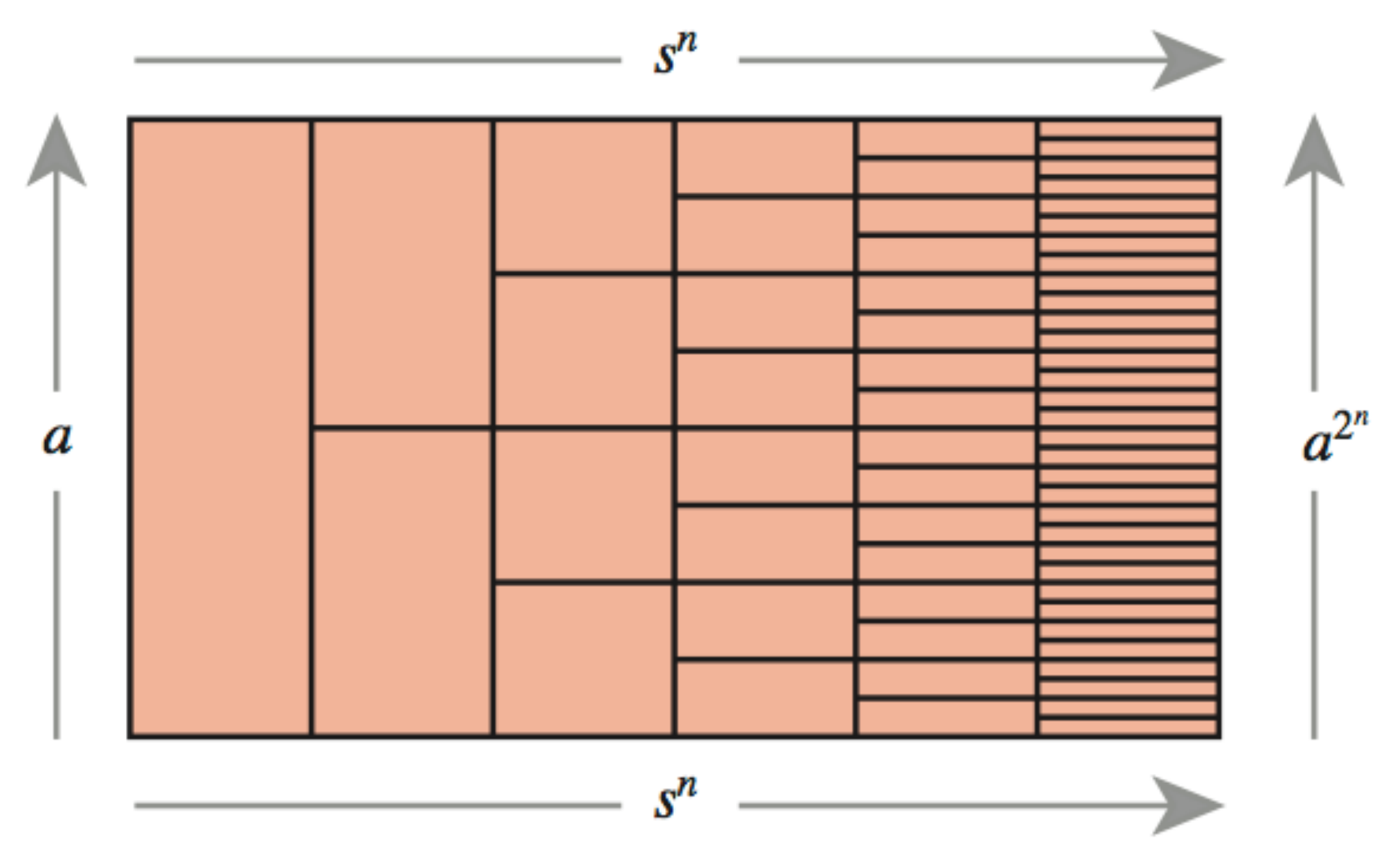}
\end{center}
\caption{A van Kampen diagram for ${a^{2^n}s^{-n}a^{-1}s^n}$ over $\Q_2$}
\label{f:bs}
\end{figure}

\begin{example}\label{ex3} We now consider an iterated form of ${\rm{BS}}(1,k)$, given by the presentation
$$\Q_{m,k}\equiv\langle a, s_1, \cdots, s_m \mid {s_1}^{-1} a {s_1} =a^k, \ {s_{i+1}}^{-1} s_i s_{i+1} ={s_i}^k \ (i >1) \rangle.$$
The Dehn function of this group was analysed by Gersten \cite{smg} (page 219), and he was the first to describe
the following diagrams. In this example we will concentrate on a family of words that arise in Section \ref{s:seed}.

In the previous example,  repeated conjugation by $s$ distorted the cyclic subgroup $\<a\>$; conjugation by
$s_1$ does this now. But now we can also distort $\<s_1\>$ by conjugating with powers of  $s_2$, and then
 distort $\<s_2\>$ by conjugating with powers of  $s_3$, and so on. In this way, we obtain van Kampen diagrams of
the type show in figure \ref{f:iterate} (where we've taken $k=2$).
 Here, each of the delineated regions of the diagram
is a scaled copy of the diagram from figure \ref{f:bs}. Conjugation by successive $s_i$ adds more layers to the diagram, culminating in a diagram with
$m$ layers. If we scale things so that  the segments labelled $s_m$ on the boundary are single edges, then we obtain
a van Kampen diagram $E_m$ whose boundary label is $a^{\Delta_k(m)}V_m$, where $\Delta_k$
is the function $\Delta_k(n)=k^{\Delta_k(n-1)}$ (with $\Delta_k(0)=k$)
and the word $V_m$ is obtained from $s_1^{-\Delta_k(m-1)}a^{-1}s_1^{\Delta_k(m-1)}$
 by first replacing the two subwords  $s_1^{\pm\Delta_k(m-1)}$ with $s_2^{-\Delta_k(m-2)}s_1^{\pm 1}s_2^{\Delta(m-2)}$, then 
replacing $s_2^{\pm\Delta_k(m-2)}$ with $s_3^{-\Delta_k(m-2)}s_2^{\pm 1}s_3^{\Delta_k(m-2)}$, and so on until each letter
$s_i$ appears only twice in $V_m$.

As in the previous example, we take a slighted shifted reflection of $E_m$ and join it to $E_m$ to obtain a 
van Kampen diagram with boundary label $W_{m,1} := aV_ma^{-1}V_m$. As Gersten \cite{smg} points out, this diagram
embeds in the universal cover of $\Q_{m,k}$, so it is least area, by Lemma \ref{aspher}. As in the 
previous example, we count the
2-cells along the segment where $E_m$ meets its reflection to obtain a lower bound of $\Delta_k(m)$ on $\ar(W_{m,1})$.
Once again, $\ar^*(W_{m,1}) = \ar(W_{m,1})$.
\end{example}

\begin{figure}
\begin{center}
\includegraphics[width=3.5in]{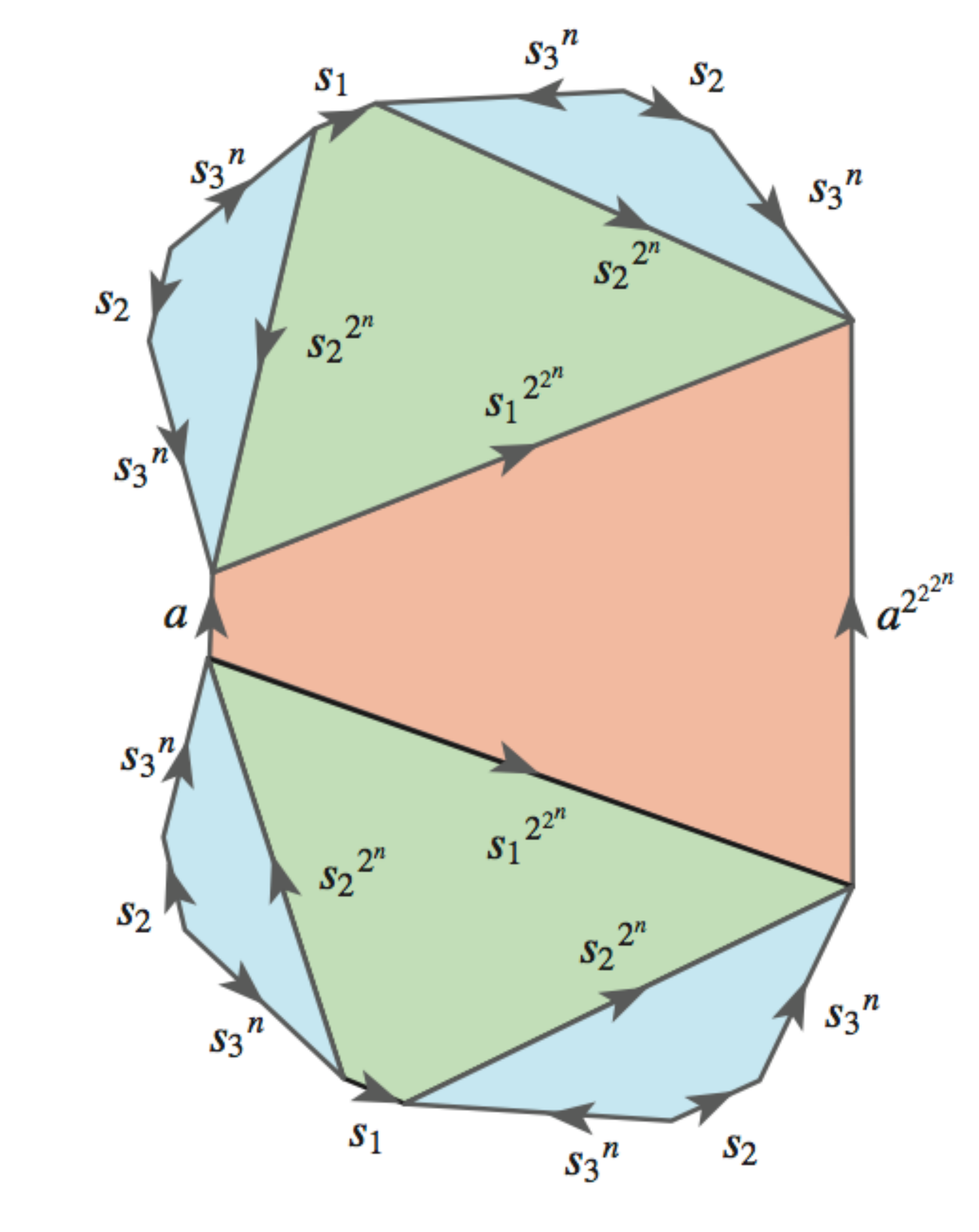}
\end{center}
\caption{A van Kampen diagram over $\Q_{2,m}$ with $m\ge 3$}
\label{f:iterate}
\end{figure}

\subsection{Aspherical Presentations} 
Recall that a group presentation $\AR$ is said to be {\em aspherical}
if the standard 2-complex of the presentation
has a contractible universal covering.

\begin{definition}
If $D$ is a van Kampen diagram over a presentation $\AR$, then
there is a unique label-preserving
combinatorial map from $D$ to the presentation complex $K$
of $\AR$, and since $D$ is 1-connected this map lifts to the universal cover
 $D\to\widetilde K$. One says that $D$ is an {\em embedded
diagram} if this lift is injective on the union of the open 2-cells
of $D$.
\end{definition}

The following  lemma is a slight variation on a
result of Gersten \cite{smg} that
 has proved extremely useful in  
 calculating lower bounds on Dehn functions
(see \cite{bfs} and references therein).

\begin{lemma} \label{gerst} Let $\AR$ be an aspherical presentation.
If $D$ is an embedded van Kampen diagram
for $w\in F(\A)$, then 
for every $n\in \mathbb Z$,
$$\ar(w^n) = |n|\, \ar(D).$$
In particular, $\ar^*(w) = \ar(D)$.
\end{lemma}

\begin{proof} A choice of base vertex $p$ identifies the 1-skeleton of $\widetilde K$  
with the Cayley graph of $\AR$ and induces a bijection
between null-homotopic words $u$
in the letters $\A^{\pm 1}$ and edge-loops $\tilde u$ based
 at $p$. We lift $D\to K$ so that the restriction to $\partial D$ of 
 $D\to\widetilde K$ parameterises $\tilde w$.

In the cellular chain complex $C_*(\widetilde K)$
of $\widetilde K$
we have $\tilde w=\partial \sum_{i=1}^N e_i$, where
the $e_i$ are the images of the 2-cells of $D$
and $N=\ar(D)$. 
More generally, $\tilde w^n=\partial \sum_{i=1}^N ne_i$.

Now, $H_2(\widetilde K,\Z)=0$,
so if $D'$ is any van Kampen diagram for $w^n$, lifted
to fill $\tilde w^n$, and if its
2-cells map  to $\e_1,\dots,\e_M$, say,
then in $C_2(\widetilde K)$ we have
$$
0=\sum_{i=1}^N ne_i - \sum_{i=1}^M \e_i.
$$
Since $C_2(\tilde K)$ is free on the set of 2-cells and the $e_i$ are distinct, we deduce that
$$\ar(D')= M \ge |n|\, N = |n|\,\ar(D).$$
\end{proof}

\subsection{Relating Area to trivialisation complexity}

The following lemma provides a crucial bridge from the study of Dehn functions to AC-complexity.

\begin{lemma}\label{l:fib} If $P\equiv\ARnn$ can be  trivialised by $m$ (dihedral) Andrews--Curtis moves, then 
$\area_P(a_i)\le \fib_m$ for $i=1,\dots,n$, where $\fib_m$ is the $m$-th Fibonacci number.
\end{lemma}

\begin{proof} Let $\A^*$ be the free monoid on the letters $a_i^{\pm 1}$.
When one makes a dihedral Andrews-Curtis move, one replaces the
list of relators
$(r_1,\dots,r_i,\dots,r_n)$ by 
$(r_1,\dots,r_i\inv,\dots,r_n)$  or 
$(r_1,\dots,u\inv r_iu,\dots,r_n)$  or
$(r_1,\dots,r_iu\inv  r_j^{\pm 1}u,\dots,r_n)$  
for some $i\neq j$ and some $u\in\A^*$. 
If  we formally make such moves and
perform no free reduction, then  at the end of a sequence of  moves we will
have replaced   $\{r_1,\dots,r_n\}$ by $\{\rho_1,\dots,\rho_n\}$ where each
$\rho_i$ is identically equal to a concatenation of conjugates of the $r_j$. 
The  number of conjugates in each concatenation after $m$ moves is no greater than the
sum of the  two greatest numbers  at the previous stage; an obvious induction proves that this is bounded above
by $\fib_m$.   

Thus if $(\rho_1,\dots,\rho_n)=(a_1^{\pm 1},\dots,a_n^{\pm 1})$, then each $a_i$ is freely equal to
a  product of at most $\fib_m$ conjugates of the $r_j^{\pm 1}$.
\end{proof}

The preceding estimate remains valid if we allow the stabilisation move (AC4) and the Tietze move (T).

\begin{prop}\label{p:tietze} If $P\equiv\ARnn$ can be  trivialised by performing $m$ moves from the
list (AC1) to (AC4) and (T), then 
$\area_P(a_i)\le \fib_m$ for $i=1,\dots,n$, where $\fib_m$ is the $m$-th Fibonacci number.
\end{prop}

\begin{proof} As in the lemma, a sequence of $m$ moves will, if no free reductions are made, result
in an equality showing that $a_i$ is freely equal to a product of at most $\fib_m$ conjugates of relations,
but now the equality will be in a larger free group $F(a_1,\dots,a_n,x_1,\dots,x_\ell)$ and some of the
relations may be $x_j^{\pm 1}$. By deleting all occurrences of the symbols $x_i$ from this equality we
obtain an equality in the free group on the $a_i$, as required. 

The insertion and deletion of empty relations via move (T) adds to the count $m$ of elementary moves but
has no effect on the size on the number of conjugates in the equality that is used to obtain an upper bound on $\ar(a_i)$.
\end{proof}

\section{The Proof of Theorem \ref{t:If}}\label{s:lower} 

When animated with the
construction of the seed groups in Section \ref{s:seed}, Theorem \ref{t:if} provides the main content of this paper.
The proof given here is self-contained except for a result concerning the area of null-homotopic words in 
colimits of diagrams of groups that involve non-injective morphisms (Lemma \ref{areaL}); this will be proved
in Section \ref{s:area}.

\begin{notation} Let $\A=\{a_0,\dots,a_n\}$. Given a word $w$ in the letters $\A^{\pm 1}$,
we write $\-w$ for the word obtained from $w$ by deleting all occurrences of $a_0^{\pm 1}$.
\end{notation}

\begin{thm}\label{t:if} Let $\P\equiv\<a_0,a_1,\dots,a_n\mid r_1,\dots, r_n\>$.
Assume
$\-\P\equiv \<a_1,\dots,a_n\mid \-r_1,\dots, \-r_n\>$ is a presentation of the trivial group
and that $a_1$ has infinite order in $|\P|$,
Then for all words $w$ in the letters $a_i^{\pm 1}$, 
\begin{enumerate}
\item $\G_w$ is trivial if  $w=1$ in $|\P|$.
\item If $w=1$ and $\-\P$ is AC-trivialisable, then
$\P_w$ is AC-trivialisable  and 
$$\log\area^*_\P(w)\le
3\, \acc(\P_w). $$
\end{enumerate}
\end{thm}

\begin{remark}  Note the self-feeding nature of the hypothesis and conclusion: given
$\P_w$, one can introduce a new generator $\tau$, replace the relations $\rho$
of $\P_w$ with words $\tilde\rho$ obtained from $\rho$ by inserting
letters $\tau^{\pm 1}$, and repeat the basic construction with this new presentation as seed.

I do not see how, by a cunning iteration of this process, one might arrive
at a seed group with an unsolvable word problem, but if one could then it
would follow from the lower bound in Theorem \ref{t:main}(2) that the
triviality problem for balanced presentations was
unsolvable, likewise the word problem, and the problem of deciding
whether a balanced presentation was AC-trivialisable.

It would then follow that the problem
of recognising the 4-sphere among PL-presentations of homology
4-spheres was algorithmically unsolvable. And likewise, the problem
of recognising whether or not a finite 2-complex was contractible would be algorithmically unsolvable.
\end{remark}

\subsection{Input data and technical lemmas}

We begin with a group presentation   $\P\equiv\<a_0,a_1,\dots,a_n\mid r_1,\dots, r_n\>$ such that
$\-\P\equiv\<a_1,\dots,a_n\mid \-r_1,\dots, \-r_n\>$ is a presentation of the trivial
group. Note that the abelianization of the group
$|\P|$ is $\mathbb Z$, generated by the image of $a_0$.

We consider a family of presentations indexed by words $w$ in
the free group on $\A$:
$$
\<\A,t\mid r_1,\dots,r_n, ta_0t^{-1}w^{-1}\>.
$$
Let $G_w$ and $\hat G_w$
 be two copies of the group given by this presentation. Let $\G_w$ be the quotient
of $G_w\ast \hat G_w$ by the normal closure of $\{t\hat a_1\inv ,\, \hat t a_1\inv \}$. 
Note that
$\G_w$ admits the balanced presentation
$$\widetilde\P_w:=
\< \A, \hat\A, t, \hat t,\,\mid \R,\, \hat\R,\,ta_0t^{-1}w\inv ,\, \hat t\hat a_0\hat t^{-1}\hat w\inv ,\,
 t\hat a_1\inv ,\, \hat t a_1\inv \>,
$$
where $\R = \{r_1,\dots,r_n\}$ and where a hat on an object denotes a second (disjoint) copy of
that object. The last two relations displayed exhibit the redundancy of the generators
$t$ and $\hat t$, which we remove to obtain a more concise balanced presentation for $\G_w$:
$$
\P_w:=
\< \A, \hat\A \mid \R,\, \hat\R,\,\hat a_1a_0\hat a_1^{-1}w\inv ,\, a_1\hat a_0 a_1^{-1}\hat w\inv \>.
$$

\begin{lemma} \label{l1}If $w=1$ in the group $|\P|$, then $\G_w\cong|\P_w|$ is trivial.
\end{lemma}

\begin{proof} If $w=1$ in $|\P|$, then $w=1$ in $G_w$. And since $a_0$ is conjugate
to $w$ in $G_w$, it follows that $a_0=1$ and hence the homomorphism $|\P|\to G_w$
implicit in the labelling of generators factors through $|\P|\to |\-\P|$. But the group $|\-\P|$
is assumed to be trivial, so $a_i=1$ in $G_w$ for $i=0,\dots, n$. Thus $G_w=\<t\>$ is
infinite cyclic. Similarly $\hat a_i=1$ in $\hat G_w$ and  $\hat G_w=\<\hat t\>$.  
The relations $t=\hat a_1$ and $\hat t = a_1$ in $\G_w$ kill  the generators
$t$ and $\hat t$.
\end{proof}

We shall need the following comparison of 
area between $\widetilde\P_w$ and $\P_w$.

\begin{lemma}\label{l:noHat} If a word $v$ in the letters $\A\cup \hat\A $
equals $1\in\G_w$, then
$$
\ar_{\P_w}(v) \le 
\ar_{\widetilde\P_w}(v) \le 
3\,\ar_{\P_w}(v).
$$
\end{lemma}

\begin{proof} We retract  the free
group with basis $\A\cup\hat\A\cup\{t,\hat t\}$ onto
the free group with basis $\A\cup\hat\A$  by defining $t\mapsto a_1$
and $\hat t\mapsto \hat a_1$. The image under this retraction of
each defining relation of $\widetilde\P_w$ is a defining relation of $\P_w$, hence the
area of $v$ over $\P_w$ is no greater than its area over  $\widetilde\P_w$.   

Conversely,  
if $N=\ar_{\P_w}(v)$ then in the free group on $\A\cup\hat A$ there is
an equality of the form
$$
v\freely \prod_{i=1}^N x_ir_i^{\pm 1}x_i^{-1}
$$
where the $r_i$ are defining relations of $\P_w$. From this we obtain an
equality expressing $v$ as a product of  conjugates of the defining relations of
$\widetilde\P_w$ by replacing each occurrence of the relation $\hat a_1a_0\hat a_1^{-1}w\inv$
by
$$
(t\hat a_1^{-1})^{-1}\ (ta_0t^{-1}w\inv)\ w(t\hat a_1^{-1})w\inv
\freely
\hat a_1(t^{-1}t)a_0(t^{-1}w\inv w t)\hat a_1^{-1}w\inv,
$$
and by making a  similar substitution for each occurrence
of $a_1\hat a_0 a_1^{-1}\hat w\inv$.  
\end{proof}

In addition to the preceding lemmas, our proof of Theorem \ref{t:if} relies
on the following special case of Theorem \ref{t:area}.  This is the only point in the proof at which we need
to assume $a_1$ has infinite order in $|\P|$.

\begin{lemma}\label{areaL} With the above notation, $
\area_{\widetilde\P_w}(a_0)\ge \area^*_\P(w).$
\end{lemma}

Together with Lemma \ref{l:noHat} this yields: 

\begin{cor} \label{c:area}$
\area_{\P_w}(a_0)\ge \frac 1 3 \area^*_\P(w).$
\end{cor}

\subsection{Proof of Theorem \ref{t:if}}
Item (1) is immediate from Lemma \ref{l1}.

(2). From Lemma \ref{l:fib} we
know that $\area_{\P_w}(a_0)\le \fib_m$, where $m=\acc(\P_w)$. 
The Fibonacci numbers $\fib_m$ satisfy $\F_m< e^m$. 
In the light of Corollary  \ref{c:area}, it follows that
$$\acc(\P_w) \ge \log\area_{\P_w}(a_0) \ge  \log\area^*_\P(w) - \log 3.$$ 
And since both sides are integers, we can replace $\log 3$ by $1$. 
\hfill$\square$

\section{A Complementary Upper Bound}\label{s:upper} 

In this section we establish an upper bound that complements Theorem \ref{t:if} and thus
complete the proof of Theorem \ref{t:main}. We maintain the notation of the previous section.  

To complement Lemma \ref{l1} we have:

\begin{lemma}\label{l2} If, in $|\P|$, both $w$ and $a_1$ have infinite order  and $\<a_1\>$ intersects 
$\<w\>$ and $\<a_0\>$ trivially,   then  $\G_w$ is infinite.
\end{lemma}

\begin{proof} Since $w$ and $a_0$
have infinite order, $G_w$ is an HNN extension of
$|\P|$. We have assumed that $\<a_1\>$ intersects the amalgamated subgroups
of this HNN extension trivially, so by Britton's Lemma \cite{LS} the subgroup
$\<a_1,t\>\subset G_w$
is free of rank 2. Thus $\G_w$ is an amalgamated free product $G_w\ast_{F_2}
\hat G_w$.
\end{proof}

The following lemma will be needed when we  compare $\P_w$ to $\-\P$. 

\begin{lemma}\label{l:elim} If $r$ is a reduced word
that contains $|r|_0$ occurences of  $a_0^{\pm 1}$,
then any presentation $\<a_0,a_1,\dots,a_n\mid \mathcal{X},r,a_0\>$ can be transformed to
$ \<a_0,a_1,\dots,a_n\mid \mathcal{X},\- r, a_0\>$ by making
at most $|r|_0$ $\acc$-moves.  
\end{lemma}

\begin{proof}  If $r=ua_0^{\pm 1}v$, we reduce $|r|_0$ by making
 the $\acc$-move $r\mapsto 
r(v^{-1}a_0^{\mp 1}v)$.
\end{proof} 

\bigskip 
\begin{quote} {\bf Summary of Input:}
We have $\P \equiv \<\A\mid\R\>\equiv\<a_0,a_1,\dots,a_n\mid r_1,\dots, r_n\>$ such that
$\<a_0\>$ and $\<a_1\>$ are infinite   and intersect trivially in $|\P|$.
 We are assuming
$\-\P\equiv \<a_1,\dots,a_n\mid \-r_1,\dots, \-r_n\>$ is a presentation of the trivial group. We consider
words $w$ in the free group
on $\A=\{a_0,a_1,\dots,a_n\}$
 such that   $\<w\>$ is trivial or infinite in $|\P|$ and $\<w\>\cap\<a_1\> = \{1\}$.
\end{quote}

\begin{thm} \label{t:main}
With input as above, 
\begin{enumerate}
\item[{\bf (1)}] $\G_w$ is trivial if and only if $w=1$ in $|\P|$.\smallskip

\item [{\bf (2)}] Suppose $w=1$ in $|\P|$. If $\-\P$ is AC-trivialisable then so is
$\P_w$, and 
$$\log\area^*_\P(w)\, -\, 1\ \le\ 
\acc(\P_w)\ \le\ 2\acc(\-\P) + 
2\area^*_\P(w)+2\sum_{r\in\R}|r|_0$$
\end{enumerate}
\end{thm}
 
\begin{proof} 
Item (1)  is an immediate consequence of Lemmas \ref{l1} and \ref{l2} and the lefthand inequality in (2) was
established in the previous section. To establish the remaining inequality, we consider a word $w$ with
 $w=1$ in $|\P|$. This means that there is an equality in the free group
$F(\A)$  
$$
w=\prod_{i=1}^N x_i\inv  r_i ^{\e_i}x_i,
$$
with $x_i\in F(\A),\, \e_i=\pm 1,\, r_i\in\R^{\pm 1}$ and $N=\area_\P(w)$. Whence the free
equality
$$
\hat a_1 a_0 \hat a_1\inv  = (\hat a_1 a_0 \hat a_1\inv  w\inv ) w=  (\hat a_1 a_0 \hat a_1\inv  w\inv )
\prod_{i=1}^N x_i\inv  r_i ^{\e_i}x_i.
$$
This equality provides a scheme for
 replacing the relation $\rho\equiv (\hat a_1 a_0 \hat a_1\inv  w\inv )$ of $\P_w$
by the relation $a_0$: first
apply $N=\area_\P(w)$ dihedral $\acc$-moves, multiplying $\rho$ on the right by
the given
conjugates $x_i\inv  r_i ^{\e_i}x_i$
of the relations $r_i\in\R$; then  conjugate by $\hat a_1$.
Similarly, one can replace $a_1\hat a_0 a_1^{-1}\hat w\inv $ by $\hat a_0$. After
doing so, repeated applications
of Lemma \ref{l:elim} allow one to delete all occurences of $a_0$ from the relations
$\R$ and all occurences of $\hat a_0$ from the relations $\hat\R$; the total
number of $\acc$-moves required to do so is  at most $2\sum |r|_0$.

At this stage we have shown that if $w=1$
then $\P_w$ is AC-equivalent to
 $\<\A,\hat\A \mid  \-\R,\, \hat{\-\R},\, a_0,\,\hat a_0\>$, and the
hypothesis that $\-\P$
is AC-trivialisable tells us that we can now perform AC-moves
not involving the letters $\{a_0,\hat a_0\}$ to transform this
to  $\<\A,\hat\A \mid  \A,\, \hat{\A}\>$.
  
A simple accounting of the
moves that we made in the above proof shows that  
$$
\acc (\P_w)\le 2\acc(\-\P) + 2\area_\P(w) +2 \sum_{r\in\R}|r|_0.
$$
\end{proof}

\section{The Area of Words in Push-Outs}\label{s:area}

The purpose of this section is to establish results that
relate the area of (null-homotopic) words in the generators
of a group presentation $\P\equiv\AR$ to the area of the same
words in augmented presentations. An understanding of this
relationship  plays a crucial role in our strategy for
obtaining lower bounds on AC-complexity. (In Section \ref{s:lower} this understanding entered in the guise of Lemma \ref{areaL}.)

Throughout this section
 we shall be careful to  retain subscripts 
to indicate which
presentation is being used to calculate
area: thus $\ar_\P(w)$ is the area of a least-area van
Kampen diagram for $w\in F(\A)$ over $\P$. 

We shall have reason to discuss words $w$ in the
alphabet $\A$ that are not null-homotopic with
respect to $\P$ but are null-homotopic with respect
to an augmentation of $\P$. For this reason it is convenient
to define 
$$
\arP(w) := +\infty \hbox{ if $w\neq 1$ in the group $|\P|$.}
$$
We shall also need to consider  
$$
\arP^*(w) := \min\{\arP(w^n) \mid n\text{ a non-zero integer}\}.
$$
Our objective is to prove:

\begin{thm}\label{t:area} Consider a finite presentation $\P\equiv\AR$, fix  $a_1\in\A$ and $u_0,u_1\in F(\A)$,  suppose that
$\langle a_1\rangle\cap \langle u_i\rangle =\{1\}$ in $|\P|$ for $i=0,1$,
and that $a_1$ has infinite order.
Let
$$\T = \langle \A,\, \hat\A,\, t,\, \hat t \mid \R, \, \hat \R,\, 
t^{-1}u_0tu_1^{-1},\, 
\hat t^{-1}\hat u_0\hat t\hat u_1^{-1},\, a_1\hat t^{-1},\, \hat a_1t^{-1}\rangle.
$$
Then, for all $v\in F(\A)$, 
$$
\ar_\T(v) \ge \min\{\arP(v),\, \arP^*(u_0),\, \arP^*(u_1)\}.
$$
\end{thm}

By taking $v=u_0$, we obtain the special case needed in 
Lemma \ref{areaL}.

\begin{cor} If  $\<u_0\>$ and $\<a_1\>$ are infinite  in
$|\P|$  and intersect trivially, 
then $\ar_\T(u_0)\ge \arP^*(u_1)$.
\end{cor}

The presentation $\T$ in Theorem \ref{t:area}
 is obtained from $\P$ in two
steps: first one forms  $\langle \A,t\mid \R,\,
tu_0t^{-1}u_1^{-1}\rangle$, then one fuses two
copies of the resulting group by identifying $a_1$
in the first copy with $t$ in the second and {\em{vice versa}}.
We consider the effect of these two operations separately.

\subsection{Pushouts of HNN type}\label{s:hnn}

Let $B$ be the group with presentation $\P\equiv\AR$
and fix $b_0,b_1\in B$. As usual
we write $F(X)$ to denote the free group on a set
$X$. Let $G$ be the pushout\footnote{colimit in the
category of groups} of the diagram
$$ B \overset{\phi_1}\longleftarrow
 F(x,y)\overset{\phi_2}\longrightarrow
F(x,t),$$
where $\phi_1(x)=b_0,\, \phi_1(y)=b_1$ and
$\phi_2(x)=x,\, \phi_2(y)=t^{-1}xt$. If $u_0,u_1\in F(\A)$
are equal in $B$ to $b_0, b_1$, respectively, then $G$
has presentation
$$
\mathcal G = \langle \A, \,t \mid \R, \, t^{-1}u_0tu_1^{-1}\rangle.
$$
Note that if the orders of $b_0,b_1\in B$ are not the
same, then the natual map $B\to G$ will not be an
injection (cf.~Remark \ref{r:hnn}).

\begin{lemma} \label{l:trings}
For all $v\in F(\A)$ one has
$$
\ar_\mathcal G (v) \ge \min \{\arP(v),\, \arP^*(u_0),\, \arP^*(u_1)\}.
$$
\end{lemma}

\begin{proof}  If $v\neq 1$ in $G$ there is nothing to
prove, so suppose $v=1$ in $G$ and consider a  least-area
van Kampen diagram $D$  over $\mathcal G$ with boundary label
$v$. If no   2-cell in $D$ has boundary
label $t^{-1}u_0tu_1^{-1}$, then $D$ is  a
diagram over $\P$ and hence has area at least $\arP(v)$.

It $D$ does contain  2-cells with boundary label $t^{-1}u_0tu_1^{-1}$,
then the union of these 2-cells form a
collection of {\em $t$-rings} in the sense of \cite{bfs}.

A $t$-ring is a subdiagram obtained as follows.
Starting in the interior of  a 2-cell $e_0$
labelled $t^{-1}u_0tu_1^{-1}$,
one crosses a 1-cell labelled $t$ to enter the interior of a 2-cell $e_1$ 
which also has boundary label  $t^{-1}u_0tu_1^{-1}$;
there is a second edge
labelled $t$ in boundary cycle of $e_2$, and crossing the
1-cell carrying that label brings one to a third 2-cell
with boundary label $t^{-1}u_0tu_1^{-1}$; continuing in
this manner one obtains a chain of 2-cells, each with boundary label $t^{-1}u_0tu_1^{-1}$.
 Because  no 1-cells of $\partial D$ are labelled $t$, this
 chain of 2-cells must close to form a {\em $t$-ring}, i.e.
the  union of the interiors of the 2-cells and the edges
labelled $t$ form an open annulus in $D$, the closure of
which is called an
{\em{annular subdiagram}}. This annular subdiagram  has 
two boundary cycles (an inner one and an outer one),
which need not be embedded. One of these
cycles is labelled by a word over the alphabet
$\{u_0, u_0^{-1}\}$ and the other
is  labelled by the same word over the alphabet
$\{u_1, u_1^{-1}\}$. This word cannot be freely equal
to the empty word because otherwise one could 
replace the subdiagram enclosed by the outer boundary cycle of 
the $t$-ring with a van Kampen diagram of zero area,
 contradicting the hypothesis that $D$ is a least-area
diagram.

Consider a $t$-ring in $D$ that is innermost, i.e. a ring
whose
inner boundary cycle $\gamma$ encloses  a subdiagram of $D$
that contains no edges labelled $t$. This subdiagram $D_0$ is a 
van Kampen diagram over  $\P$ for the word labelling $\gamma$,
which is   freely equal to $u_0^n$ or $u_1^n$
for some $n\neq 0$. Thus the area
of $D_0$, and hence of $D$, is at least $\min\{\arP^*(u_0),
\arP^*(u_1)\}$. 
\end{proof}

\begin{remark} \label{r:hnn} In the preceding lemma,  $B\to G$ is injective
if and only if $b_0,b_1\in B$ have the same order. The well-known
but non-trivial
``if" implication can be
 proved by arguing that if $v\in F(\A)$ equals to $1\in G$ then there is 
diagram for $v$ containing no $t$-edges, and hence $v=1$ in $B$.
Indeed, in any diagram $D$ for $v$ with a $t$-edge,  there would be a $t$-ring 
and
an innermost such $R$ would
enclose a van Kampen diagram $D'$ over $\P$ for a word
freely equal to $u_0^n$ or $u_1^n$, where
$n\neq 0$. But $D'$ shows that one (hence both) of
 $b_0^n$ and $b_1^n$ equals $1\in B$; in other
words  $u_0^n=u_1^n=1$ in $B$. By deleting $R$ and $D'$ from $D$
and replacing them
with a van Kampen
diagram over $\P$, one reduces the number of  $t$-edges in $D$.

 A slight modification of this argument yields Britton's Lemma \cite{LS}.
\end{remark}

We shall be most interested in the following special case of
Lemma \ref{l:trings}.

\begin{cor} \label{cor1} If $b_0=u_0$ has infinite order in $B$, then
$$\arP(u_1)\ge \ar_\mathcal G(u_1) \ge \arP^*(u_1).$$
\end{cor}

The hypothesis in Theorem \ref{t:area} that $a_1$ has infinite
order in $|\P|$ is included in order to admit the following lemma.

\begin{lemma}\label{lem2} If $a_1\in\A$ has infinite order in $B$
and $\langle a_1\rangle \cap \langle b_0\rangle
= \langle a_1\rangle \cap \langle b_1\rangle = \{1\}$,
then for every   non-trivial word $v\in F(a_1,t)$,
$$
\ar_\mathcal G(v) \ge \min \{ \arP^*(u_0),\, \arP^*(u_1)\}.
$$
\end{lemma}

\begin{proof} As in the proof of Lemma \ref{l:trings}, we
will be done if we can argue that any van Kampen diagram
$D$ for $v$ over $\mathcal G$ must contain a $t$-ring. Since
$a_1$ has infinite order in $B$, this is clear if $v$
is of the form $a_1^n$, so we may assume that
$v$ is a reduced word that contains at least one occurence
of $t$. And by induction on the length of $v$ we may assume
that $D$ is a non-singular disc.

At each edge of $\partial D$ there begins a $t$-corridor\footnote{Like a $t$-ring, a
$t$-corridor is a chain of 2-cells joined along $t$-edges; but instead of closing-up,
a $t$-corridor begins and ends at $t$-edges on the boundary of $\partial D$. See \cite{bfs}.},
in the sense of \cite{bfs}. We
focus our attention on an outermost $t$-corridor, i.e. a
$t$-corridor whose initial and terminal $t$-edges lie 
at the ends of an arc $\alpha$ of $\partial D$ labelled $a_1^n$.
The side of the corridor joining the endpoints of $\alpha$
is labelled  $u_j^m$ (where $j=0$ or $1$, and $m\neq 0$).
 The existence of the subdiagram $D'\subset
D$ bounded by  this side
and $\alpha$  proves that
$u_0^m= a_1^n$ in $G$. But by hypothesis, $u_0^m=b_0^m
\neq a_1^n$ in $B$. Therefore $D'$ is not a diagram over
$\P$ and hence must contain a $t$-edge in its interior. It
follows that $D'$ contains a $t$-ring, since $\partial D'$ has no edges labelled $t$.
\end{proof}

\subsection{Pushouts of twisted-double type}\label{s:abc}

Let $H$ be a group and let $\G$ be the pushout of the diagram
$$ 
H \overset{\psi_1}\leftarrow F(x,y)\overset{\psi_2}\rightarrow H,
$$
where $\psi_1(x)=\psi_2(y)$ and $\psi_1(y)=\psi_2(x)$. 
Fix a presentation $\Q=\langle\B\mid\mathcal S\rangle$ for $H$ such that $\B$
contains letters  $\beta_1, \beta_2$  with $\psi_1(x)=\beta_1$ and
$\psi_1(y)=\beta_2$ in $\G$. Then, using hats to denote a second (disjoint)
copy of each set and symbol, 
we have the following presentation of $\G$
$$
\U \equiv \langle \B, \hat\B 
\mid \mathcal{S}, \,\hat{\mathcal{S}},\, 
 \beta_1\hat\beta_2^{-1},\,
\hat\beta_1\beta_2^{-1}\rangle.
$$
Note that if the exchange $\beta_1\leftrightarrow\beta_2$ does not 
induce an isomorphism of 
$\langle\beta_1,\beta_2\rangle\subset H$,
 then the natural map $H\to\G$ 
will not be an injection.

\smallskip
Let $\Psi = \min\{\ar_\Q (w) \mid w\in F(\beta_1,\beta_2)\ssm\{1\}\}$.

\begin{lemma}\label{lem3} For every $v\in F(\B)$,
$$
\ar_\U(v) \ge \min \{ \ar_\Q(v),\, \Psi\}.
$$
\end{lemma}

\begin{proof} If $v\neq 1$ in $\G$ then there is nothing to prove. If
$v=1$ in $\G$ then we consider a least-area van Kampen diagram
$D$ for $v$. 

We need two observations concerning the geometry of $D$.
First, since ${\mathcal S}$-labelled
and $\hat{\mathcal S}$-labelled 2-cells have no edges in common, and since
no $\hat{\mathcal S}$-labelled 2-cell has an edge on $\partial D$, each
connected component of the frontier of  the union of the  $\hat{\mathcal S}$-labelled 2-cells 
 determines a non-empty
chain of 2-cells in $D$  labelled  $\beta_1\hat\beta_2^{-1}$ or
$\hat\beta_1\beta_2^{-1}$. In particular, if $D$ has  no 2-cells labelled
$\beta_1\hat\beta_2^{-1}$ or
$\hat\beta_1\beta_2^{-1}$, then it has no $\hat{\mathcal S}$-labelled 2-cells either.

Conversely, any {\em{reduced}} diagram with boundary label
in $F(\B)$ that has no $\hat{\mathcal S}$-labelled 2-cells  cannot have
any 2-cells  labelled  $\beta_1\hat\beta_2^{-1}$ or $\hat\beta_1\beta_2^{-1}$, 
because the absence of $\hat{\mathcal S}$-labelled 2-cells
would force the existence of a cancelling pair of faces at any edge labelled
$\beta_i$.

 These observations mean that we have only two cases to consider:
either all of the 2-cells of $D$ are labelled by relations from ${\mathcal S}$,
in which case  $D$ is a diagram over $\Q$
and $\ar_\T(v)=\ar_\Q(v)$; or else
 $D$  contains 2-cells with labels  from $\hat{\mathcal S}$.

In the latter case, we focus our attention on an innermost
component of the frontier of  the union of  the $\hat{\mathcal S}$-labelled 2-cells.
This defines a  chain of 2-cells  labelled  $\beta_1\hat\beta_2^{-1}$ or
$\hat\beta_1\beta_2^{-1}$ that encloses a van Kampen diagram 
over $\Q$ (or $\hat\Q$) whose boundary cycle 
is labelled by a 
word in the letters $\beta_1, \beta_2$ (resp. 
$\hat\beta_1, \hat\beta_2$). Thus, in this case,
$\ar\, D\ge \Psi$.

(The careful reader may worry that the  innermost component
we were just considering yielded
a chain $c$ of 2-cells with labels $\beta_i\hat\beta_j^{-1}$ that encloses
a subdiagram of zero area. But this case cannot arise, because if it did then
one could excise $c$ and [noting that its outer
boundary cycle would  be freely equal to the empty word] replace
it with a zero-area  subdiagram, thus contradicting the assumption that $D$
is a least-area diagram.) \end{proof}

\subsection{Proof of Theorem \ref{t:area}} The presentation $\T$ in the
theorem is obtained by applying the process of subsection \ref{s:hnn} to
$\AR$ and then the process of subsection \ref{s:abc} to the resulting
presentation, with $a_1$ and $t$ in the r\^oles of $\beta_1$ and $\beta_2$.

 Lemma \ref{lem2} bounds the quantity $\Psi$ in Lemma \ref{lem3}:
$$\Psi\ge \min\{\arP^*(u_0),\, \arP^*(u_1)\}.$$
The theorem then follows immediately from
Lemma  \ref{l:trings} and Lemma  \ref{lem3} . \hfill $\square$

\section{Seed Groups}\label{s:seed}

\def\ar{\area}

\def\arP{\area_\P}

Our purpose in this section is to animate the
Main Construction (Theorem \ref{t:main}) with
examples. We shall focus in particular
on the group $S_2$. 
This group has a long history \cite{baum}, \cite{higman}. It's isoperimetric
properties were first studied by Gersten \cite{smg}
and later by Platonov \cite{platonov}. 
It belongs to the
following family.

\subsection{The Groups $S_k$}

The main examples that we shall consider are the groups
$S_k\ (k\ge 2)$ with presentation
$$
\S_k\equiv \langle x,t \mid (t x t\inv)x (t x t\inv)\inv = x^k
\rangle.
$$

\begin{definition}  Fix $k>1$.
The function $\Delta_k:\N\to\N$ is defined
recursively by $\Delta_k(0)=k$ and $\Delta_k(n+1)=k^{\Delta_k(n)}$.
\end{definition}

In the context of the current article, the key property of $S_k$ is  
the following. 

\begin{thm} \label{thmS} There exists a sequence of words $w_n\in F(x,t)$
with lengths $|w_n|\le 12n$ such that $w_n=1$ in $S_k$ and
$\ar_{\S_k}^*(w_n) \ge \Delta_k(\lfloor\log_2 n\rfloor)$.
\end{thm}

\begin{remark} 
$S_k$ is torsion-free  and  $x\in S_k$ is non-trivial.
If one deletes all occurences of $t$ from the relations of $\S_2$, one obtains 
$\-\S_2 \equiv \langle x\mid xx^{-2}\rangle\equiv \langle x\mid x\rangle$.
Thus, casting $t$ in the r\^ole of $a_0$ and $x$ in the r\^ole of $a_1$,
we see that Theorem  \ref{thmS} provides the input necessary 
to deduce Theorem \ref{thmA} from Theorem \ref{t:If}. 
\end{remark}

The proof of Theorem \ref{thmS} occupies the remainder of this section.
The words $w_n$ (which do not depend on $k$) are defined as follows.

\begin{definition}\label{d:w_n} If 
$n=2^m$, then $w_n:=xV_mx^{-1}V_m^{-1}$, where  $V_n$ is defined recursively by the rule
$$
V_0 = x  
\ \hbox{and  }\ V_m = t V_{m-1}t\inv x t V_{m-1}^{-1}t^{-1}.
$$
If $2^{m} < n<2^{m+1}$, then $w_n:=w_{2^m}$.
\end{definition}

\begin{lemma} \label{wnL} For all $k,n\in\mathbb N$,
\begin{enumerate}
\item[{\bf (1)}] 
$|w_{n}|\le 12n-8$,
\item[{\bf (2)}] $w_{n}= 1$ in $S_k$.
\end{enumerate}
\end{lemma}

\begin{proof} An induction on $m$ shows that $|V_m| = 2^m6-5$ and 
$V_m=x^{\Delta_k(m)}$ in $S_k$. 
Hence  $|w_{2^m}| =  2^{m+1}6-8$ and  $w_{2^m}= 1$ in $S_k$.
\end{proof}
  
\begin{remark}
Theorem \ref{thmS} confirms that 
$\Delta_k(\lfloor\log_2 n\rfloor)$ is a lower 
bound on
the Dehn function of $S_k$.  Platonov \cite{platonov} showed (for $k=2$) that, up to $\simeq$ equivalence,  it  is
also an upper bound.
\end{remark}

\subsection{Outline of the proof} Each of the words $w_n$ labels an edge-loop in the universal cover of the
standard 2-complex $K=K(\S_k)$ and we seek a lower bound on the area (number of 2-cells) in any van Kampen diagram
filling this loop. The idea of the proof is as follows: first we push the loop (and any disc $D$ filling it)
down to the infinite cyclic covering $L$ of $K$; we then shrink a tree in $L$ to produce a 1-vertex complex
that is the standard 2-complex of a natural presentation of the kernel of 
the retraction $S_k\to\<t\>$; we retract this complex onto an aspherical subcomplex containing the image of our loop,
which is now labelled by the word $\check w_n^\dagger$ of Lemma \ref{tilde0}; in the universal cover of 
this subcomplex, a lift of our loop bounds an embedded disc, and the number of 2-cells in this disc (which we recognise
from Example \ref{ex3}) gives the desired lower bound on the area of the original disc $D$.  

\subsection{The approximating groups $B_m$} 

\def\ar{\area}

\def\arP{\area_\P}

We fix $k\ge 2$ and for each positive integer $m$
consider the group $B_m$ with
presentation
$$
\B_m \equiv \langle  x_0,\dots, x_m \mid 
x_{i+1} x_ix_{i+1}\inv = x_i^k \text{  for }i=0,\dots,m-1\rangle.
$$

Let $\iota:B_m\to S_k$ be the  homomorphism that sends $x_i\in B_m$
to $t^ixt^{-i}\in S_k$. We shall see that these maps are injective.

\begin{lemma} \label{Binf}
The kernel of the retraction $\pi:S_k\to \langle t \rangle$ has presentation
$$
\B_\infty \equiv \langle x_i\ (i\in\Z) \mid x_{i+1}x_i x_{i+1}\inv x_i^{-k} \ (i\in\Z)\rangle.
$$
More precisely, the map $x_i\mapsto t^ixt^{-i}$ defines a monomorphism $B_\infty\to S_k$ with image $\ker \pi$.
\end{lemma}

\begin{proof} 
The kernel of $\pi$ is the normal 
closure of $x$, and it is helpful
to view it as the fundamental group of the 
infinite cyclic covering   $L$ of the standard 2-complex $K(\S_k)$. 
The 1-skeleton $L^{(1)}$
of $L$ consists of a line of directed edges
labelled $t$ with a loop labelled $x$ at each vertex;
there is an edge-circuit
labelled $(t xt\inv )x(tx\inv t^{-1})x^{-k}$ beginning 
at each vertex and $L$ is
obtained from $L^{(1)}$ by attaching   a 
2-cell to each each of these circuits.

To obtain a homotopy equivalence $h$ from $L$ to the standard 2-complex  
 $K(\B_\infty)$, one shrinks the line of $t$-edges
 in $L^{(1)}$ to a point and sends the loop
at the $n$th vertex of $L$ to the directed edge of $K(\B_\infty)$ labelled $x_n$; one then extends
 the map to 2-cells in the obvious manner.
\end{proof}  

\begin{lemma} \label{l:ntoinfty} The map $B_n\to B_\infty$ implicit in the
labelling of generators is injective. Moreover,
$\ar_{\B_n}(w) = \ar_{\B_\infty}(w)$ for all $w\in \Fn$.
\end{lemma}

\begin{proof} These facts follow easily from the observation
that killing the generators $x_i$ with $i<0$
gives a retraction from  $B_\infty$ to the subgroup generated by $\{x_n\ (n\in\N)\}$, which
has presentation $\langle x_n\,(n\in\N)\mid
x_{n+1} x_nx_{n+1}\inv x_n^{-k} \, (n\in\N)\rangle$. This subgroup is obtain from $B_n$ by forming
repeated HNN extensions along infinite cyclic subgroups, and the inclusion of the base group into such an
HNN extension does not distort area.
\end{proof}

The area estimates that we will need in the groups $B_m$ were hinted at in Subsection \ref{ss:dehn}. They
rely on the following elementary lemma, which  is well known. 

\begin{lemma} If the presentation $\AR$ of $G$ is aspherical,
and $u,v\in F(\A)$ define elements of infinite order in $G$,
then the presentation $\langle \A, t \mid \R,\, t\inv utv\inv
\rangle$ is aspherical.
\end{lemma}

We constructed $B_m$ as an iterated HNN extension of an infinite
cyclic group, with cyclic amalgamations at each stage, so by
 iterated application of the   lemma we have:

\begin{cor}\label{aspher} The presentation $\B_m$ is aspherical.
\end{cor}

\subsection{Good preimages of words}

Let $F_\infty$ be the free group on the set $\{x_m: m\in\Z\}$ and
consider the homormorphism $\phi:F_\infty\to F(x,t)$
defined by $x_m\mapsto t^mxt^{-m}$.  
If $w\in F(x,t)$ has exponent
sum zero in $t$, then the following set is non-empty
$$
\Phi (w) =\{ \check w \in \F_\infty \mid \phi(\check w) = w\}.
$$

\begin{lemma}\label{tilde0}
 There exists $\check w^\dagger\in \Phi(w)$ with $|\check w^\dagger| \le |w|$, satisfying the following properties:
 \begin{enumerate}
 \item if $u=v^{-1}$ then $\check u^\dagger = (\check v^\dagger)\inv$ ;
 \item if $u=tvt^{-1}$ then $\check u^\dagger = \sigma(\check v^\dagger)$, where $\sigma:F_\infty\to F_\infty$ is the
 automorphism $x_i\mapsto x_{i+1} \ \forall i$ ;
 \item if $u=v_1v_2$ where the $v_i$ have exponent sum zero in $t$, then $\check u^\dagger = \check v_1^\dagger\check v_2^\dagger$.
 \end{enumerate}
\end{lemma}

\begin{proof} One obtains $\check w^\dagger$ from $w$ as follows: 
place brackets $[$ and $]$ at the
beginning and end of $w$; then, 
reading from the left, replace each letter $x^{\pm 1}$
by  the string  $$x^{\pm 1} t^{-m}].[t^{m}$$ where $m$ is the exponent sum in
$t$ of the prefix read so far; then the content of each bracket is freely equal to a word of the form
 $[t^{r}x^{\pm 1}t^{-r}]$, which we replace  by $x_r^{\pm 1}$.  We then delete all brackets.
\end{proof}

\begin{prop}\label{sameA} If  $w\in F(x,t)$ equals the identity in $S_k$, then
\begin{enumerate}
\item $
\area_{\S_k}(w) \le \area_{\B_\infty}(\check w)$ for all $\check w\in\Phi(w)$;
\item $
\area_{\S_k}(w) = \area_{\B_\infty}(\check w^\dagger)$.
\end{enumerate}
\end{prop}

\begin{proof}  Lemma \ref{Binf} tells us
that if $w=1$ in $S_k$ then $\check w=1$ in $B_\infty$ for all $\check w\in\Phi(w)$.
Hence there exists an equality in $F_\infty$ of the form
$$
\check w \freely \prod_{j=1}^A u_jr_{i(j)}^{\e(j)}u_j^{-1}
,$$
where $A = \ar_{\B_\infty}(\check w)$, with
$i(j)\in \Z$ and 
$r_i \equiv x_{i+1}x_ix_{i+1}\inv x_i^{-k}$, and
$\e(i)=\pm 1$.

In  $F(x,t)$ we consider the images under $\phi:F_\infty\to F(x,t)$ of the terms on
both sides of this equality: writing $U_j$ for the image of $u_j$ 
and noting that the image of $r_i$
is $t^i(txt\inv)x(txt\inv)\inv x^{-k}t^{-i}$, we get
$$
w \freely \prod_{j=1}^A (U_jt^{i(j)})\rho^{\e(j)}(U_jt^{i(j)})^{-1},
$$
where $\rho \equiv (txt\inv)x(txt\inv)\inv x^{-k}$. Thus 
$\ar_{\Sigma_k}(w)\le A= \ar_{\B_\infty}(\check w)$. This proves (1).

\smallskip
We shall prove (2) topologically using the homotopy equivalence 
$h:L\to K(\B_\infty)$ described in the proof of Lemma \ref{Binf}.
Consider
 the loop $\lambda(w)$ in $L^{(1)}$ that begins
at the vertex $0$  and is labelled $w$. The edge-loop 
 in $ K(\B_\infty)$  that is the image of $\lambda(w)$ under $h$ 
is labelled $\check w^\dagger\in F_\infty$.

Let $D$ be a least-area van Kampen diagram
 for $w$ over $\S$ and consider the unique label-preserving combinatorial
map $D\to L$ whose restriction to $\partial D$ is $\lambda(w)$.  
The composition of this map 
with  $h$ gives a van Kampen
diagram $\widehat{D}$ for $\check w^\dagger$ over 
the presentation $\B_\infty$; and since $h$ 
is a homeomorphism on the complement of the 1-skeleton,
$\ar(\widehat{D})=\ar(D)$. Hence
$$
\ar_{\S_k}(w)=\ar(D)=\ar(\widehat{D})
\ge \ar_{B_\infty}(\check w^\dagger),
$$
complementing the inequality in (1).
\end{proof}

\subsection{The Required Area Estimate}

Suppose $2^m\le n< 2^{m+1}$ and let
$w_n= xV_mx^{-1}V_m^{-1}$ be as in Definition \ref{d:w_n}.

\begin{lemma}\label{l:dagger}
\begin{enumerate}
\item $ \check V_0^\dagger=x_0$
\item $\check V_{m}^\dagger = \sigma(\check V_{m-1}^\dagger)\, x_0\, \sigma(\check V_{m-1}^\dagger)^{-1}$
\item $\check w_n^\dagger = x_0\check V_{m}^\dagger  x_0^{-1}(\check V_{m}^\dagger)^{-1}.$
\end{enumerate}
\end{lemma}

\begin{proof} This is immediate from the inductive definition of $V_m$ and the properties of 
the assignment $w\mapsto \check w^\dagger$
described in Lemma \ref{tilde0}.
\end{proof}

The words $\tilde V_{m}^\dagger$ and $\tilde w_n^\dagger$
 involve only the letters $x_0,\dots,x_m$.
The sequence of words  $\check V_1^\dagger, \check V_2^\dagger, \check V_3^\dagger, \dots$ begins
$x_1 x_0 x_1^{-1},$ then
$$
(x_2 x_1 x_2^{-1}) x_0 (x_2 x_1 x_2^{-1})^{-1},
$$
$$
(x_3 x_2 x_3^{-1}) x_1 (x_3 x_2 x_3^{-1})^{-1} x_0 (x_3 x_2 x_3^{-1})^{-1} x_1^{-1} (x_3 x_2 x_3^{-1})
$$

As in Lemma \ref{wnL}, a simple induction establishes:

\begin{lemma}  
\begin{enumerate}
\item  $\tilde V_{m}^\dagger$ has length $2^{m+1}-1$ and $w_{2^m}$ has length $2^{m+2}$.
\item $\check V_{m}^\dagger = x_0^{\Delta_k(m)}$ in $B_m$, and $\check w_n^\dagger =1$.
\end{enumerate}
\end{lemma}

In fact, the words $\check w_{2^m}$ are precisely the words $W_{m,1}$ described in Example \ref{ex3}
(modulo a renaming of letters). As we noted
there, a well known argument due to Gersten  (Example 6.3 of \cite{smg})
 shows that each of these words bounds an embedded diagram in the universal
cover of the standard 2-complex of $\B_m$, and the area of the diagram $D_m$ for $W_{m,1}$ has area greater than\footnote{In Gersten's notation, $\Delta_2(m)=E_m(1)$.}
$\Delta_k(m)$ (see also \cite{bfs}, Exercise 7.2.11). In summary:

\begin{lemma} \label{discs} Let $2^m\le n < 2^{m+1}$.
Over the presentation $\B_m$
there is an embedded van Kampen
diagram $D_m$ for $\check w_n^\dagger$ and
$$\area(D_m) \ge \Delta_k(m).$$
\end{lemma}

\noindent{{\bf{Proof of Theorem \ref{thmS}:}} In the light of Lemma \ref{wnL},
it only remains to prove that 
$\ar_{\S_k}^*(w_n) \ge \Delta_k(\lfloor\log_2 n\rfloor)$. Let $m=\lfloor \log_2(n)\rfloor$.
In  Proposition \ref{sameA} we proved that
$
\area^*_\S(w_n) = \area^*_{\B_\infty}(\tilde w_n^\dagger)$ and in Lemma \ref{l:ntoinfty}
we proved that this was equal to
$\area^*_{\B_m}(\tilde w_n^\dagger)$, since all the letters of $\tilde w_n^\dagger$ lie
in $\{x_0,\dots,x_m\}$. 

In Lemma \ref{discs} we found an embedded van Kampen diagram $D_m$ for
$\tilde w_n^\dagger$ over the presentation $\B_m$. The presentation
$\B_m$ is aspherical (Corollary \ref{aspher}), so by Lemma \ref{gerst},
$$
\area^*_{\B_m}(\tilde w_n^\dagger) = \area (D_m) \ge \Delta_k(\lfloor\log_2 n\rfloor),
$$
and the proof is complete.
\hfill $\square$
 
 \section{Closing Remarks}\label{s:last}

Whenever one it concerned with the number of elementary moves that are required to transform one
mathematical object into another, it can be helpful to regard the objects as vertices of a graph,
with an edge connecting a pair of vertices that differ by an elementary move. (If the elementary moves are not invertible,
one considers a directed graph.) 

For example, one might consider the {\em Tietze graph} $\mathfrak{T}$, consisting 
of all finite presentations over a fixed countable alphabet, with edges corresponding to Tietze moves.
This has one component for each isomorphism class of finitely
presented groups. The
unsolvability
of the {\em triviality problem} for arbitrary finite presentations  
translates into a statement about the
lack of coarse connectedness for the sub-level sets of the function $\lambda$ that assigns to
a presentation (vertex) the sum of the lengths of its relators. For example, given any recursive function
$f:\mathbb N\to \mathbb N$, for sufficiently large $m$, there are vertices $P$ in the path component
of $\mathbb{I}_1 = \<a\mid a\>$ such that $\lambda(P)=m$ but $P$ cannot be connected to $\mathbb{I}_1$
in $\lambda^{-1}[0,f(m)]$.

\smallskip

In what follows we shall write $\Lambda(m) = \lambda\inv [0,m]$ for sub-level sets of $\lambda$
(in whatever graph of presentations we are considering).

\subsection{Andrews-Curtis Graphs}  For each positive integer $k$, the  
 Andrews-Curtis moves (AC1) to (AC3) define a graph whose vertices are the balanced presentations of the trivial
 group (considered as quotients of  a fixed free group $F_k$); we denote this graph $\ac_k$. 
By introducing the stabilisation move (AC4) one connects each vertex of
$\ac_k$ to a vertex in $\ac_{k+1}$, and it is natural to consider 
$\ac_\infty$, the union of these graphs.
In this language, the AC-conjecture is the assertion that each of the graphs $\ac_k$ is connected, and the 
stable AC-conjecture is that $\ac_\infty$ is connected. 
 Theorem \ref{thmA} can be viewed as an insight into
the coarse Morse theory of $\ac_\infty$ with respect to the height function $\lambda$:

\begin{thm}\label{t:hint} Let $k\ge 4$.
There is a constant $C$ and  a sequence  of vertices $P_n\in\Lambda(Cn)$,
all in the path component of $\ac_k$ containing $\mathbb{I}_k$, so that $P_n$ cannot be
connected to $\mathbb{I}_k$ by a path in $\Lambda(\Delta(\lfloor\log_2 n\rfloor))\subset\ac_\infty$.
\end{thm}  
 
Several authors have considered Andrews-Curtis equivalence for presentations of groups other than the trivial
group, e.g. \cite{BKM}, \cite{BLM}, \cite{DG}, \cite{myr}. This amounts to an exploration of
different components of the graph whose
vertices are all finite presentations over a fixed countable alphabet, with edges corresponding to
the (stable or unstable) AC-moves.

\begin{remark} In a subsequent article I shall 
explain how the construction behind our main theorem allows one to extend the work of 
A. Nabutovsky and S. Weinberger, as surveyed in \cite{schmuel}. They exploit
group-theoretic complexity to explore the sub-level sets of functionals such as diameter on moduli spaces of metrics for closed manifolds in dimensions greater than $4$. Our constructions allow one to extend parts of their work to dimension $4$. Theorem \ref{t:hint} is in the spirit of such results.
\end{remark}

 \subsection{The rank 2 case}
The techniques that we have 
developed in this article do not provide any 
information  about $\ac_2$, but Boris Lishak \cite{lishak} recently proved a result similar to Theorem \ref{thmA} in the rank 2 case. His
techniques are different to ours, but he too uses the Baumslag-Gersten group.
This is particularly interesting in the light of what is known about small neighbourhoods of the basepoint
$\mathbb{I}_2\in\ac_2$. Casson showed that 
the image of  $\Lambda(13)$ in $\Lambda(25)$ is contained in the union
of at most 6 path components, and that the radius of   $\Lambda(13)$ in the path-component of $\mathbb{I}_2$
is at most $217$. The main result of \cite{mm} states that  $\Lambda(12)\subset AC_2$ is connected.
Casson also proved that the binary icosahedral group is the
only non-trivial perfect group that has a 2-generator balanced presentation where the total length of the relations
is at most $13$.

\subsection{Polynomial Time Algorithms}\label{ss:polynomial}
We noted in the introduction that even in situations where it is physically impossible to write down a trivialising
sequence of AC-moves, there might still be a polynomial time algorithm that can determine the existence of AC-trivialisations. This is closely akin to the fact that a group can admit a polynomial time solution to the
word problem even if the Dehn function of the group is huge \cite{BORS}.
In this spirit, 
Diekert {\em et al.} \cite{DLU} used data compression techniques to prove that
 the word problem in $S_2$, the seed group from Section \ref{s:seed}, can be solved in cubic time (cf.~\cite{MUW}).
 It follows that if one builds the presentations $\P_w$ in Theorem \ref{t:Main} based on words in the generators
 of $S_2$, then the AC-triviality of $\P_w$ can be determined in cubic time, even though the number
 of AC-moves needed in any trivialisation grows like $\Delta_2(\log_2 |w|)$.
 
\subsection{An Example}\label{example}

Let me close by writing down an explicit presentation to emphasize that the explosive growth in the length
of AC-trivialisations begins with relatively small presentations. Here is a balanced presentation
of the trivial group that requires more than $10^{10000}$
AC-moves  to trivialise it. We use the commutator convention $[x,y]=xyx^{-1}y^{-1}$.
\begin{align*}
\< a,t, \alpha, \tau \mid  [tat^{-1},a]a^{-1},&\ \  \ \ [\t\a\t^{-1},\a]\a^{-1},\\
& \a t^{-1}\a^{-1}  [a,\, [t[t[ta^{20}t^{-1},\,a]t^{-1},\, a]t^{-1},\, a]],\\
&\ \ \ \ \ \ \ \ \ \ \ \  a \t^{-1}a^{-1}  [\a,\, [\t[\t[\t\a^{20}\t^{-1},\,\a]\t^{-1},\, \a]\t^{-1},\, \a]]\>. 
\end{align*} 
\medskip

\noindent{\bf{Acknowledgement:}} The figures in this paper were drawn by Tim Riley. I am grateful to him for letting
me use them.

\end{document}